\documentclass[final,3p]{elsarticle}
\usepackage{amsmath}
\usepackage{amssymb}
\usepackage{amsthm}

\usepackage{mathrsfs}
\usepackage[color=Purple!50!white]{todonotes}
\newtheorem{thm}{Theorem}[section]
\newtheorem*{thm*}{Theorem}

\newtheorem{prop}[thm]{Proposition}
\newtheorem{claim}[thm]{Claim}
\newtheorem{cor}[thm]{Corollary}

\newtheorem{qu}[thm]{Question}
\theoremstyle{definition}

\newcommand{\flo}[1]{\lfloor #1 \rfloor}
\newcommand{\cei}[1]{\lceil #1 \rceil}

\tikzset{global scale/.style={
		scale=#1,
		every node/.append style={scale=#1}
	}
}

\allowdisplaybreaks
\begin{document}

\begin{frontmatter}



\title{Ramsey goodness of fans}


\author[inst1]{Yanbo Zhang}
\ead{ybzhang@hebtu.edu.cn}

\affiliation[inst1]{organization={School of Mathematical Sciences, Hebei Normal University},%Department and Organization
            %addressline={}, 
            city={Shijiazhuang},
            postcode={050024}, 
            state={Hebei},
            country={China}}
\author[inst2]{Yaojun Chen}
\ead{yaojunc@nju.edu.cn}

\affiliation[inst2]{organization={Department of Mathematics, Nanjing University},%Department and Organization
            %addressline={Address Two}, 
            city={Nanjing},
            postcode={210093}, 
            state={Jiangsu},
            country={China}}

\begin{abstract}
Given two graphs $G_1$ and $G_2$, the Ramsey number $r(G_1,G_2)$ refers to the smallest positive integer $N$ such that any graph $G$ with $N$ vertices contains $G_1$ as a subgraph, or the complement of $G$ contains $G_2$ as a subgraph. A connected graph $H$ is said to be $p$-good if $r(K_p,H)=(p-1)(|H|-1)+1$. A generalized fan, denoted as $K_1+nH$, is formed by the disjoint union of $n$ copies of $H$ along with an additional vertex that is connected to each vertex of $nH$. Recently Chung and Lin proved that $K_1+nH$ is $p$-good for $n\ge cp\ell/|H|$, where $c\approx 52.456$ and $\ell=r(K_{p},H)$. They also posed the question of improving the lower bound of $n$ further so that $K_1+nH$ remains $p$-good.

In this paper, we present three different methods to improve the range of $n$. First, we apply the Andr{\'a}sfai-Erd{\H{o}}s-S{\'o}s theorem to reduce $c$ from $52.456$ to $3$. Second, we utilize the approach established by Chen and Zhang to achieve a further reduction of $c$ to $2$. Lastly, we employ a new method to bring $c$ down to $1$. In addition, when $K_1+nH$ forms a fan graph $F_n$, we can further obtain a slightly more refined bound of $n$.
\end{abstract}



\begin{keyword}
Ramsey goodness \sep Fan \sep Complete graph
\MSC[2020] 05C55 \sep 05D10
\end{keyword}

\end{frontmatter}


\section{Introduction}

We use $K_p$ to denote a complete graph on $p$ vertices. Given two graphs $G_1$ and $G_2$, we use the notation $K_N\to (G_1,G_2)$ to indicate that for any graph $G$ with $N$ vertices, either $G$ contains $G_1$ as a subgraph or its complement, denoted as $\overline{G}$, contains $G_2$ as a subgraph. The \emph{Ramsey number $r(G_1,G_2)$} refers to the smallest positive integer $N$ such that $K_N\to (G_1,G_2)$. In the early years of graph Ramsey theory, Chv\'atal and Harary~\cite{Chvatal1972Generalized} established a useful lower bound of $r(G_1,G_2)$. Let $\chi$ be the chromatic number of the graph $G_1$ and let $c$ be the number of vertices in a largest connected component of $G_2$. The lower bound can be expressed as \[r(G_1,G_2)\ge (\chi-1)(c-1)+1\,.\] This lower bound is straightforward to obtain by considering the graph consisting of the disjoint union of $\chi-1$ copies of $K_{c-1}$. It does not contain a connected component with at least $c$ vertices; hence it does not contain $G_2$ as a subgraph. On the other hand, the chromatic number of its complement is $\chi-1$, implying that $\overline{G}$ does not contain $G_1$ as a subgraph. As a special case of the above lower bound, when $H$ is a connected graph, it is obvious that $r(K_p,H)\ge (p-1)(|H|-1)+1$, where $|H|$ represents the order of $H$. Burr and Erd\H{o}s~\cite{Burr1983Generalizations} then introduced the concept of $p$-goodness. A connected graph $H$ is called \emph{$p$-good} if \[r(K_p,H)=(p-1)(|H|-1)+1\,.\]

Based on this, Burr presented a stronger version of the lower bound and introduced the concept of $G_1$-goodness. The symbols $\chi$ and $c$ are defined as before. We use $s$ to represent the smallest number of vertices in any color class over all proper colorings of $G_1$ with $\chi$ colors. When $c \ge s$, the lower bound can be expressed as \[r(G_1,G_2) \ge (\chi-1)(c-1)+s.\] Moreover, if equality holds, $G_2$ is called \emph{$G_1$-good}. In the following, we only discuss the case where $G_1$ is a complete graph $K_{p}$. Specifically, what types of sparse graphs are considered to be $p$-good?

Chv\'atal~\cite{Chvatal1977Tree} demonstrated a simple yet fundamental result: all trees are $p$-good for every positive integer $p$. In addition to trees, cycles are among the most studied sparse graphs. Let $C_n$ denote a cycle of length $n$. Bondy and Erd\H{o}s~\cite{Bondy1973Ramsey} established that for $n \ge p^{2}-2$ and $p \ge 3$, $C_n$ is $p$-good. Schiermeyer~\cite{Schiermeyer2003All} subsequently extended this finding to hold when $n \ge p^{2}-2p$ and $p > 3$. Building upon these advancements, Nikiforov~\cite{Nikiforov2005Cycle} broadened the scope of $n$ to $n \ge 4p+2$. A recent development came from the work of Keevash, Long, and Skokan~\cite{Keevash2019Cycle}, who established the existence of an absolute constant $C$ such that for $n \ge C\frac{\log p}{\log\log p}$, $C_n$ is $p$-good. Recall that the book $B_{k,n}$ comprises $n-k$ distinct $(k+1)$-cliques sharing a common $k$-clique. Nikiforov and Rousseau~\cite{Nikiforov2004Large} utilized the regularity lemma to show that for fixed values of $p \ge 3$ and $k \ge 3$ and sufficiently large $n$, $B_{k,n}$ is $p$-good. Quite recently, Fox, He, and Wigderson~\cite{Fox2023Ramsey} made a significant improvement to this result by demonstrating that every $B_{k,n}$ with $n \ge 2^{k^{10p}}$ is $p$-good.

This paper primarily investigates the $p$-good property of generalized fans. We denote the disjoint union of $n$ copies of $H$ as $nH$. A generalized fan $K_1+nH$ consists of $nH$ along with an additional vertex, where this extra vertex is connected to each vertex of $nH$. Li and Rousseau~\cite{Li1996Fan} utilized the Erd\H{o}s-Simonovits stability lemma to establish that $K_1+nH$ is $(K_2+G_1)$-good, where $G_1$ and $H$ are arbitrary fixed graphs and $n$ is sufficiently large. When $G_1$ is a complete graph, Nikiforov and Rousseau~\cite{Nikiforov2009Ramsey} further showed a more generalized result. Given a graph $G$ of order $n$ and a vector of positive integers $\mathbf{k}=(k_1, \ldots ,k_n)$, we define $G^{\mathbf{k}}$ as the graph obtained from $G$ by replacing each vertex $i \in [n]$ with a clique of order $k_i$, and replacing every edge $ij\in E(G)$ with a complete bipartite graph $K_{k_i,k_j}$.
 
 \begin{thm}[Nikiforov and Rousseau~\cite{Nikiforov2009Ramsey}]
	Suppose $K$ and $p$ are positive integers with $p\ge 3$, $T_n$ is a tree of order $n$, and $\mathbf{k}=(k_1, \ldots ,k_n)$ is a vector of integers with $0<k_i\le K$ for all $i\in [n]$. Then, for sufficiently large $n$, $T_n^{\mathbf{k}}$ is $p$-good.
  \end{thm}

Since the above theorem utilizes Szemer\'edi's regularity lemma, the lower bound of $n$ is of tower type. The next natural idea is to expand the range of $n$. When $H=K_2$, we refer to $K_1+nK_2$ as a fan, denoted as $F_n$. Li and Rousseau~\cite{Li1996Fan} also proved that $F_n$ is $3$-good for $n\ge 2$. Surahmat, Baskoro, and Broersma~\cite{Surahmat2005Ramsey} extended this result and proved that $F_n$ is $4$-good for $n\ge 3$. The authors of this paper~\cite{Zhang2014Ramsey} developed a method to prove that $F_n$ is $5$-good for $n\ge 5$. Based on the same approach, Kadota, Onozuka, and Suzuki~\cite{Kadota2019graph} subsequently established that $F_n$ is $6$-good for $n\ge 6$.

Recently, Chung and Lin~\cite{Chung2022Fan} considered the general case and derived the following result.
\begin{thm}[Chung and Lin~\cite{Chung2022Fan}]\label{ChungLin}
Let $H$ be an arbitrary given graph, $h=|H|$, $\ell=r(K_p,H)$, and $c=(3+3\sqrt{2})^{2}\approx 52.456$. Then $K_1+nH$ is $p$-good for $n\ge cp\ell/h$.
\end{thm}

They also raised the question of further improving the lower bound of $n$. This paper aims to answer the question. Our main theorem is as follows.
\begin{thm}\label{main}
Let $H$ be an arbitrary given graph, $h=|H|$, and $\ell=r(K_p,H)$. Then $K_1+nH$ is $p$-good if one of the following conditions is satisfied:
\begin{enumerate}[(1)]
\item $n\ge 2p\ell/h$;
\item $n\ge \max\{(p+2)\ell/h, \ell-h\}$.	
\end{enumerate}
\end{thm}

We observe that the first condition above effectively lowers the requirement from $n\ge 53p\ell/h$ in Chung and Lin's result to $n\ge 2p\ell/h$. In the second condition, we further decrease the requirement from $n\ge 2p\ell/h$ to $n\ge \max\{(p+2)\ell/h, \ell-h\}$. We will prove this theorem in Sections~\ref{secstructure} and \ref{secgeneral}. Before demonstrating that both of these conditions ensure $K_1+nH$ to be $p$-good, we will provide a little appetizer in Section~\ref{secgeneral}: we can reduce the constant $c$ from $53$ to $3$ by utilizing the Andr{\'a}sfai-Erd{\H{o}}s-S{\'o}s theorem. These proof methods may provide assistance in further expanding the range of $n$.

As a corollary to Theorem~\ref{ChungLin}, $F_n$ is $p$-good for $n\ge ((3+3\sqrt{2})p)^2/2\approx 26.228p^2$. In Section~\ref{secfan}, we will improve this lower bound to $n\ge (p^2-p-2)/2$ and $p\ge 3$.

\begin{thm}
$F_n$ is $p$-good for $n\ge (p^2-p-2)/2$ and $p\ge 3$.
\end{thm}

Wang and Qian~\cite{Wang2022Ramsey} studied the Ramsey number $r(K_p,s(K_1+nK_t))$, where $s(K_1+nK_t)$ represents the disjoint union of $s$ copies of $K_1+nK_t$. They proved that $r(K_p,s(K_1+nK_t))=nt(p+s-2)+s$ for sufficiently large $n$. By utilizing Theorem~\ref{main}, we can immediately obtain the following corollary, which improves upon their result. We consider the disjoint union of $K_{(nh+1)s-1}$ and $(p-2)K_{nh}$, which provides a lower bound for this result. The upper bound is obtained by iterating the result of Theorem~\ref{main}.

\begin{cor}
	Let $s$ and $p$ be positive integers. Let $H$ be an arbitrary connected graph with $h=|H|$ and $\ell=r(K_p,H)$. Then $r(K_p,s(K_1+nH))=nh(p+s-2)+s$ if one of the following conditions is satisfied:
	\begin{enumerate}[(1)]
	\item $n\ge 2p\ell/h$;
	\item $n\ge \max\{(p+2)\ell/h, \ell-h\}$.
	\end{enumerate}
\end{cor}

Based on the existing results, Chung and Lin~\cite{Chung2022Fan} proposed several interesting research questions. One of these questions is finding a characterization for graphs $G$ that are $K_p$-good. In other words, what kind of graph $G$ satisfies $r(K_p,G)=(p-1)(|G|-1)+1$? Let us consider the other aspect of the problem: what kind of graph $G$ satisfies $K_1+nH$ is $G$-good for sufficiently large $n$? Specifically, we propose the following problem.

\begin{qu}
	Determine all the graphs $G$ such that $r(G,K_1+nH)=n|H|(\chi(G)-1)+1$ for sufficiently large $n$.
\end{qu}

In the final part of this section, we introduce some terminology and notation. We use $[p]$ to denote the set $\{1, 2, \ldots, p\}$. If a graph $G$ does not contain $G_1$ as a subgraph and its complement does not contain $G_2$ as a subgraph, then the graph $G$ is called a $(G_1, G_2)$-Ramsey graph. The independence number of a graph $G$ is denoted by $\alpha(G)$. The set of neighbors of a vertex $v$ in the graph $\overline{G}$ is denoted as $N_{\overline{G}}(v)$, and $d_{\overline{G}}(v)=|N_{\overline{G}}(v)|$. The subgraph induced by $V(G)\setminus U$ in $G$ is denoted as $G-U$. The symbol $e(U_1, U_2)$ represents the number of edges with one endpoint in $U_1$ and the other endpoint in $U_2$. When $U_1$ contains only one vertex $u_1$, it is denoted as $e(u_1, U_2)$.

%%%%%%%%%%%%%%%%%%%%%%%%%%%%%%%%%%%%%%%%%%
\section{The structure of a $(K_p,K_1+nH)$-Ramsey graph}\label{secstructure}

Our structural analysis in this section follows a similar approach to that of Zhang and Chen~\cite{Zhang2014Ramsey}. To prove the upper bound $r(K_p,K_1+nH)\le (p-1)nh+1$, we proceed with an induction on $p$. The theorem is trivial for $p=1,2$. Assume $p\ge 3$ and that the upper bound holds for smaller values. Let $G$ be a graph with $(p-1)nh+1$ vertices, where \[n\ge \min\{2p\ell/h, \max\{((p+2)(\ell-h)-2)/h, \ell-h, p-1\}\}.\] We proceed by contradiction and assume that $G$ is a $(K_p,K_1+nH)$-Ramsey graph. Now we analyze the properties of the graph $G$. Firstly, we claim that the degree of every vertex in $\overline{G}$ falls within a small range.
\begin{prop}\label{property1}
	For every vertex $v$, we have $nh\le d_{\overline{G}}(v)\le (n-1)h+\ell-1$.
\end{prop}
\begin{proof}
	To prove the first inequality, suppose that $d_{\overline{G}}(v)\le nh-1$. By removing $v$ and all its neighbors in $\overline{G}$, we are left with at least $(p-2)nh+1$ remaining vertices. Let us denote the induced subgraph on these remaining vertices as $H_1$. By the inductive hypothesis, $H_1$ contains $K_{p-1}$ as a subgraph, which together with $v$ forms a $K_p$ in $G$, contradicting our assumption.
	
	Regarding the second inequality, suppose that $d_{\overline{G}}(v)\ge (n-1)h+\ell$. We denote the subgraph induced by all vertices not adjacent to $v$ as $H_2$. It follows from $r(K_p, H)\le \ell$ that $r(K_p, nH)\le (n-1)h+\ell$. So $\overline{H_2}$ contains $nH$ as a subgraph. Consequently, we obtain a $K_1+nH$ in $\overline{G}$ with $v$ as its hub, once again contradicting our assumption.
\end{proof}
From the above property we see that $nh\le (n-1)h+\ell-1$ which implies \[\ell\ge h+1,\] since otherwise our proof is done.

\begin{prop}\label{property2}
	$\alpha(G)\le nh-1$.
\end{prop}
\begin{proof}
	If not, there exists a $K_{nh}$ in $\overline{G}$, which we denote as $H_3$. The induced subgraph $G-V(H_3)$ has $(p-2)nh+1$ vertices. According to the inductive hypothesis, it contains $K_{p-1}$ as a subgraph. We denote this subgraph as $H_4$ and its vertices as $x_1,\ldots, x_{p-1}$. Each vertex of $H_3$ is not adjacent to at least one of $x_1,\ldots, x_{p-1}$, since otherwise $G$ contains a $K_p$. Hence, there exists an $i$ with $i\in [p-1]$ such that $|V(H_3)\setminus N_G(x_i)|\ge nh/(p-1)\ge h$. We choose any vertex from $V(H_3)\setminus N_G(x_i)$ as a hub, and $h-1$ vertices in $V(H_3)\setminus N_G(x_i)$ together with $x_i$ form an $H$ in $\overline{G}$. The other vertices of $H_3$ contains $n-1$ disjoint copies of $H$ in $\overline{G}$. Consequently, $\overline{G}$ contains a $K_1+nH$, a contradiction.
\end{proof}

From the induction hypothesis, it can be observed that $G$ contains $K_{p-1}$ as its subgraph. Let us denote this subgraph as $H_5$, and its $p-1$ vertices as $u_1, \ldots, u_{p-1}$. For each $i$ with $i\in [p-1]$, we use $U_i$ to represent the set of vertices that are not adjacent to $u_i$ but are adjacent to every other vertex of $H_5$. Here, we assume that $u_i\in U_i$ for simplicity of notation. Let $U_p$ denote the set of vertices that are not adjacent to at least two vertices of $H_5$. It is easy to see that $U_1, \ldots, U_p$ form a partition of $V(G)$.

\begin{prop}\label{property3}
	For each $i\in [p-1]$, the graph induced by $U_i$ has no edges.
\end{prop}
\begin{proof}
	If not, say, $U_1$ has an edge with two ends $y_1$ and $y_2$. Then $\{y_1,y_2,u_2,u_3,\ldots,u_{p-1}\}$ forms a $K_p$ in $G$, resulting in a contradiction.
\end{proof}

\begin{prop}\label{property4}
	$|U_p|\le (p-1)(\ell-h)-1$ and $nh-(p-1)(\ell-h-1)+1\le |U_i|\le nh-1$ for $i\in [p-1]$. Moreover, for each $i\in [p-1]$, we have $\sum_{k=1}^{i}|U_k|\ge i(nh-1)-(p-1)(\ell-h-1)+2$.
\end{prop}
\begin{proof}
	To prove the first inequality, we count twice the number of edges between $V(H_5)$ and $V(G)\setminus V(H_5)$ in $\overline{G}$. By the definitions of $U_i$ for $i\in [p]$, $\overline{G}$ contains at least $|V(G)\setminus V(H_5)|+|U_p|=(p-1)(nh-1)+1+|U_p|$ edges between $V(H_5)$ and $V(G)\setminus V(H_5)$. Moreover, Property~\ref{property1} implies that $\overline{G}$ contains at most $|V(H_5)|((n-1)h+\ell-1)=(p-1)((n-1)h+\ell-1)$ edges between $V(H_5)$ and $V(G)\setminus V(H_5)$. Thus, $(p-1)(nh-1)+1+|U_p|\le (p-1)((n-1)h+\ell-1)$ and hence $|U_p|\le (p-1)(\ell-h)-1$.
	
	Combining Properties~\ref{property2} and \ref{property3} we have $|U_i|\le nh-1$ for $i\in [p-1]$. Further, for $i\in [p-1]$,
	\begin{align*}
		|U_i| &=|V(G)|-|U_p|-\sum_{j\in [p-1],\ j\neq i}|U_j| \\
		&\ge (p-1)nh+1-((p-1)(\ell-h)-1)-(p-2)(nh-1) \\
		&=nh-(p-1)(\ell-h-1)+1.
	\end{align*}
	\begin{align*}
		\sum_{k=1}^{i}|U_k|=&|V(G)|-\sum_{k=i+1}^{p}|U_k|\\
		\ge &(p-1)nh+1-(p-1-i)(nh-1)-((p-1)(\ell-h)-1)\\
		=&i(nh-1)-(p-1)(\ell-h-1)+2.\qedhere
	\end{align*}
\end{proof}

%%%%%%%%%%%%%%%%%%%%%%%%%%%%%%%%%%%%%%%%%%
\section{Complete graphs versus generalized fans}\label{secgeneral}

Since $p-1$ disjoint copies of $K_{nh}$ contains no $K_1+nH$ and its complement contains no $K_p$, it follows that $r(K_p,K_1+nH)\ge (p-1)nh+1$. To prove the upper bound, our setup aligns precisely with Section~\ref{secstructure}: we shall employ both mathematical induction and proof by contradiction. Recall that $G$ is a $(K_p,K_1+nH)$-Ramsey graph of order $(p-1)nh+1$. We will derive a series of contradictions.

Now we divide the problem into three parts. We utilize different methods to prove that the equality $r(K_p,K_1+nH)=(p-1)nh+1$ holds for different ranges of $n$. Specifically, it holds for $n\ge 3p\ell/h$ in the first part, for $n\ge 2p\ell/h$ in the second part, and for $n\ge \max\{((p+2)(\ell-h)-2)/h, \ell-h, p-1\}$ in the last part. Since \[\max\{(p+2)\ell/h, \ell-h\}\ge \max\{((p+2)(\ell-h)-2)/h, \ell-h, p-1\},\] the last part actually establishes a broader range for $n$ compared to the range of $n$ in Condition (2) of Theorem~\ref{main}.
\subsection{$n\ge 3p\ell/h$}

By the following theorem, we can quickly deduce that $\delta(G)$ cannot be large, which contradicts Property~\ref{property1}.
\begin{thm}[Andr{\'a}sfai, Erd{\H{o}}s, and S{\'o}s~\cite{Andrasfai1974connection}]\label{thmaes}
	Let $p\ge 3$. For any graph $G$, at most two of the following properties can hold:
	\begin{enumerate}[(1).]
		\item $K_{p}$ is not a subgraph of $G$,
		\item $\delta(G)>\frac{3p-7}{3p-4}|V(G)|$,
		\item $\chi(G)\ge p$.
		\end{enumerate}
\end{thm}

From Property~\ref{property1}, we know that $\delta(G)\ge (p-2)nh+h-\ell+1$. When $n\ge 3p\ell/h$, it is easy to verify that $\delta(G)/|V(G)|>(3p-7)/(3p-4)$. Thus, the second property of Theorem~\ref{thmaes} is satisfied. Additionally, based on our assumption, the first property of the theorem is also met. According to Theorem~\ref{thmaes}, we conclude that $\chi(G)\le p-1$. Therefore, $G$ is a multipartite graph with at most $p-1$ parts. By the pigeonhole principle, there must exist a part with at least $nh+1$ vertices. These vertices in $\overline{G}$ form a complete graph which contains $K_1+nH$ as a subgraph, leading to a contradiction.

\subsection{$n\ge 2p\ell/h$}

In the previous section, we partitioned $V(G)$ into $p$ parts $U_1, \ldots ,U_p$. Now we proceed to analyze the property of $U_p$.
\begin{claim}\label{claimUp}
	For each vertex $u\in U_p$, there exists an $i\in[p-1]$ such that $e(u, U_i)=0$.
\end{claim}
\begin{proof}
	Suppose to the contrary that there exists a vertex $u\in U_p$ satisfying that $e(u, U_i)>0$ for every $i\in[p-1]$. Without loss of generality, assume that $|U_1\setminus N_G(u)|\ge |U_2\setminus N_G(u)|\ge \cdots \ge |U_{p-1}\setminus N_G(u)|$.	We will use induction to select a vertex $w_1$ from $U_1$, a vertex $w_2$ from $U_2$, and so on, until we select a vertex $w_{p-1}$ from $U_{p-1}$, such that $u,w_1,w_2,\ldots, w_{p-1}$ form a clique. Since $e(u, U_1)>0$, we can choose a vertex $w_1$ from $U_1\cap N_G(u)$. For $i\ge 2$, assuming we have found $w_1,w_2,\ldots, w_{i-1}$, we will find the required vertex $w_i$ in $U_i$.
	
	If $|U_i\setminus N_G(u)|\ge nh/i+(h-1)$, by our assumption, $|U_k\setminus N_G(u)|\ge nh/i+(h-1)$ for each $k\in [i]$. Consequently, $\overline{G}$ contains $\cei{n/i}$ disjoint copies of $H$ in $U_k\setminus N_G(u)$ for each $k\in [i]$. It follows that $\overline{G}$ contains a $K_1+nH$ with $u$ as its hub, contradicting our assumption. Thus, we assume that $|U_i\setminus N_G(u)|< nh/i+(h-1)$.
	
	For each $k$ with $k\in [i-1]$, we show that $|U_i\setminus N_G(w_k)|\le (n+2)h-2-|U_k|$. If not, $\overline{G}$ contains $\flo{(|U_k|-1)/h}$ disjoint copies of $H$ in $U_k\setminus \{w_k\}$, and $\flo{((n+2)h-1-|U_k|)/h}$ disjoint copies of $H$ in $U_i\setminus N_G(w_k)$. Since $\flo{(|U_k|-1)/h}\ge (|U_k|-1)/h-(1-1/h)$ and $\flo{((n+2)h-1-|U_k|)/h}\ge ((n+2)h-1-|U_k|)/h-(1-1/h)$, adding these two numbers together we obtain at least $n$ disjoint copies of $H$ in $N_{\overline{G}}(w_k)$. Thus, $\overline{G}$ contains a $K_1+nH$ with $w_k$ as its hub, a contradiction which implies that $|U_i\setminus N_G(w_k)|\le (n+2)h-2-|U_k|$ for each $k\in [i-1]$.
		
	Combining Property~\ref{property4}, the inequalities $n\ge 2p\ell/h$, $\ell\ge h+1$, $p\ge i+1$, and $i\ge 2$, we have:
	\begin{align*}
		&|U_i\cap N_G(u) \cap N_G(w_1)\cap \cdots \cap N_G(w_{i-1})|\\
		\ge &|U_i|-|U_i\setminus N_G(u)|-|U_i\setminus N_G(w_1)|-\cdots-|U_i\setminus N_G(w_{i-1})| \\
		> &|U_i|-(nh/i+(h-1))-((n+2)h-2-|U_1|)-\cdots-((n+2)h-2-|U_{i-1}|) \\
		=&\sum_{k=1}^{i}|U_k|-(nh/i+(h-1))-(i-1)((n+2)h-2)\\
		\ge &i(nh-1)-(p-1)(\ell-h-1)+2-(nh/i+h-1)-(i-1)(nh+2h-2)\\
		=&nh(1-1/i)-i(2h-1)-(p-1)(\ell-h-1)+h+1\\
		\ge &2p\ell(1-1/i)-i(2h-1)-(p-1)(\ell-h-1)+h+1\\
		\ge &2p(h+1)(1-1/i)-i(2h-1)+h+1\\
		\ge &2(i+1)(h+1)(1-1/i)-i(2h-1)+h+1\\
		=&h(1-2/i)+3i+1-2/i\\
		> & 0
	\end{align*}

	In this way, we can find $w_i$ in $U_i$, which is connected to $u,w_1,\ldots,w_{i-1}$. The iterative process allows us to construct a clique with $p$ vertices. Consequently, this proves that our assumption is incorrect. As a result, the claim is proven.
\end{proof}

For each vertex $u$ in $U_p$, if there exist $i,j\in [p-1]$ with $i\neq j$ such that $e(u,U_i)=e(u,U_j)=0$, then it follows from $|U_i|+|U_j|\ge 2(nh-1)-(p-1)(\ell-h-1)+2$ that  $\overline{G}$ contains a $K_1+nH$ centered at $u$, where the $n$ disjoint copies of $H$ are distributed among $U_i$ and $U_j$. This contradicts the assumption. Therefore, for each vertex $u\in U_p$, there exists a unique $i\in [p-1]$ such that $e(u, U_i)=0$. We partition $U_p$ into $p-1$ parts, $W_1,\ldots, W_{p-1}$. For every vertex $u$ of $U_p$, if $e(u, U_i)=0$, we assign $u$ to $W_i$. Since $G$ has $(p-1)nh+1$ vertices, there must exist an $i\in [p-1]$ such that $|U_i|+|W_i|\ge nh+1$. Without loss of generality, let us assume $|U_1|+|W_1|\ge nh+1$.

\begin{claim}\label{claim6}
	$W_1$ forms an independent set.
\end{claim}
\begin{proof}
If not, there must exist an edge in $G[W_1]$, denoted as $y_1y_2$. We then select one vertex each from $U_2, \ldots, U_{p-1}$, denoted as $w_2,\ldots, w_{p-1}$, respectively, such that $y_1,y_2,w_2,\ldots, w_{p-1}$ form a clique. We adopt an inductive approach, assuming that we have already found $w_2,\ldots, w_{i-1}$, and then we proceed to find $w_i$ within $U_i$. Using the same proof method as the claim above, we know that for $|U_i\setminus N_G(y_1)|\le (n+2)h-2-|U_1|$, $|U_i\setminus N_G(y_2)|\le (n+2)h-2-|U_1|$, and $|U_i\setminus N_G(w_k)|\le (n+2)h-2-|U_k|$ for each $k$ where $k\in [2,i-1]$.

Combining Property~\ref{property4}, the inequalities $n\ge 2p\ell/h$, $\ell\ge h+1$, $p\ge i+1$, and $i\ge 2$, we have:
\begin{align*}
	&|U_i\cap N_G(y_1) \cap N_G(y_2)\cap N_G(w_2)\cap \cdots \cap N_G(w_{i-1})|\\
	\ge &|U_i|-|U_i\setminus N_G(y_1)|-|U_i\setminus N_G(y_2)|-|U_i\setminus N_G(w_2)|-\cdots-|U_i\setminus N_G(w_{i-1})| \\
	\ge &|U_i|-2((n+2)h-2-|U_1|)-((n+2)h-2-|U_2|)-\cdots-((n+2)h-2-|U_{i-1}|) \\
	=&\sum_{k=1}^{i}|U_k|+|U_1|-i((n+2)h-2)\\
	\ge &i(nh-1)-(p-1)(\ell-h-1)+2+(nh-(p-1)(\ell-h-1)+1)-i(nh+2h-2)\\
	=&nh-2(p-1)(\ell-h-1)-2ih+i+3\\
	\ge &2p\ell-2(p-1)(\ell-h-1)-2ih+i+3\\
	= &2ph+2p+2\ell-2h-2ih+i+1\\
	\ge &2h+3i+5\\
	> & 0
\end{align*}

In this way, we can find $w_i$ in $U_i$, which is connected to $y_1, y_2, w_1, \ldots, w_{i-1}$. Through such iteration, we can find a clique with $p$ vertices. This demonstrates that our assumption is incorrect, and thus $W_1$ forms an independent set.
\end{proof}

Since both $U_1$ and $W_1$ are independent sets, and $e(U_1,W_1)=0$, the union $U_1 \cup W_1$ forms an independent set with at least $nh+1$ vertices. Consequently, $\overline{G}$ contains a $K_{nh+1}$ as a subgraph, which implies that $\overline{G}$ must also contain a $K_1+nH$, contradicting our assumption. This proves our theorem.

\subsection{$n\ge \max\{((p+2)(\ell-h)-2)/h, \ell-h, p-1\}$}\label{subsection3}

Recall that $U_1, \ldots, U_p$ form a partition of $V(G)$, $|U_p|\le (p-1)(\ell-h)-1$, and each $U_i$ forms an independent set for $i\in [p-1]$. Without loss of generality, assume that $|U_1|\ge |U_2|\ge \cdots\ge |U_{p-1}|$. By the pigeonhole principle, we have $|U_1|\ge (|V(G)|-|U_{p}|)/(p-1)$. Hence \[|U_1|\ge nh-\ell+h+1.\]

Now we define a subset $U'_1$ in $V(G)\setminus U_1$. For any vertex in $V(G)\setminus U_1$, if it is not adjacent to any vertex in $U_1$, then it belongs to $U'_1$. Note that $U'_1$ may be empty. If $|U'_1|\ge nh+1-|U_1|$, then we can take all the vertices in $U_1$ and $nh+1-|U_1|$ vertices in $U'_1$, which induce a subgraph in $\overline{G}$ that contains $K_1+nH$ as a subgraph. This is because each vertex in $U'_1$, together with $h-1$ vertices in $U_1$, forms an independent set, thus containing $H$ as a subgraph in $\overline{G}$. Since $n+h\ge \ell$, we have $|U_1|-1\ge n(h-1)$. It follows that $|U_1|-1\ge (h-1)(nh+1-|U_1|)$. Thus, the $nh+1-|U_1|$ vertices in $U'_1$ can form $nh+1-|U_1|$ disjoint independent sets of size $h$, together with $(h-1)(nh+1-|U_1|)$ vertices in $U_1$, which induce $(nh+1-|U_1|)H$ as a subgraph in $\overline{G}$. The remaining $|U_1|-(h-1)(nh+1-|U_1|)=(|U_1|-nh+n-1)h+1$ vertices in $U_1$ form an independent set, thus inducing $(|U_1|-nh+n-1)H$ as a subgraph in $\overline{G}$. There is one more vertex that is not adjacent to any vertex in the obtained $nH$ subgraph, and it can serve as the hub of $K_1+nH$. Thus, we obtain $K_1+nH$ in $\overline{G}$, which is a contradiction. This contradiction implies that $|U'_1|\le nh-|U_1|$. Therefore, $|U_1|+|U'_1|\le nh$.

There are at least $(p-2)nh+1$ vertices in $G-(U_1\cup U'_1)$. By the induction hypothesis, $G-(U_1\cup U'_1)$ contains $K_{p-1}$ as a subgraph. Let us denote these $p-1$ vertices as $v_1,v_2,\ldots,v_{p-1}$. In the following, we want to prove that there exists a vertex in $U_1$ that is adjacent to $v_1,v_2,\ldots,v_{p-1}$. This implies that $G$ contains $K_p$ as a subgraph.

\begin{claim}\label{claim8}
	For $k\in [p-1]$, the number of non-neighbors of $v_k$ in $U_1$ is at most $2(\ell-h)-1$.
\end{claim}

\begin{proof}
	First, $v_k$ has at least one neighbor in $U_1$, otherwise $v_k \in U'_1$, which contradicts the fact that $v_k\in V(G)\setminus(U_1\cup U'_1)$. Let $u_k$ be a neighbor of $v_k$ in $U_1$. Next, we will proceed by contradiction. If the number of non-neighbors of $v_k$ in $U_1$ is at least $2(\ell-h)$, then $|N_{\overline{G}}(v_k)\cap N_{\overline{G}}(u_k)|\ge 2(\ell-h)$. By the principle of inclusion-exclusion and Property~\ref{property1}, $|N_{\overline{G}}(v_k)\cup N_{\overline{G}}(u_k)|\le 2((n-1)h+\ell-1)-2(\ell-h)=2nh-2$. There are at least $(p-3)nh+1$ vertices in $V(G)$ but not in $\{v_k,u_k\}\cup N_{\overline{G}}(v_k)\cup N_{\overline{G}}(u_k)$, and all these vertices are adjacent to both $v_k$ and $u_k$. By the induction hypothesis, the induced graph formed by these vertices contains a subgraph $K_{p-2}$, which together with $\{v_k,u_k\}$ forms a $K_p$, a contradiction. Hence, our claim has been proven.
\end{proof}

\begin{claim}\label{claim9}
	For $k\in [p-2]$, if $v_1,v_2, \ldots, v_k$ have a common neighbor in $U_1$, then the number of vertices in $U_1$ that are not adjacent to at least one of $v_1,v_2, \ldots, v_k$ is at most $(k+1)(\ell-h)-1$.
\end{claim}

\begin{proof}
	Let us denote a common neighbor of $v_1,v_2, \ldots, v_k$ in $U_1$ as $u_k$. Now, let us proceed by contradiction. Suppose to the contrary that at least $(k+1)(\ell-h)$ vertices of $U_1$ satisfying that each of them is not adjacent to at least one of $v_1,v_2, \ldots, v_k$. By employing the principle of inclusion-exclusion and utilizing Property~\ref{property1}, it follows that $|N_{\overline{G}}(v_1)\cup \cdots\cup  N_{\overline{G}}(v_k)\cup N_{\overline{G}}(u_k)|\le (k+1)((n-1)h+\ell-1)-(k+1)(\ell-h)=(k+1)(nh-1)$. There are at least $(p-k-2)nh+1$ vertices in $V(G)$ that are not in the set $\{v_1, \ldots, v_k, u_k\}\cup N_{\overline{G}}(v_1)\cup \cdots\cup N_{\overline{G}}(v_k)\cup N_{\overline{G}}(u_k)$, and all of these vertices are adjacent to each of $v_1, \ldots, v_k, u_k$. By the induction hypothesis, the induced subgraph formed by these vertices contains a subgraph $K_{p-k-1}$, which together with $\{v_1, \ldots, v_k, u_k\}$ forms a $K_p$, a contradiction. This completes our claim.
\end{proof}
	
\begin{claim}\label{claim10}
	For $k\in [p-2]$, if the number of vertices in $U_1$ that are not adjacent to at least one of $v_1,v_2, \ldots, v_k$ is at most $(k+1)(\ell-h)-1$, and the number of non-neighbors of $v_{k+1}$ in $U_1$ is at most $2(\ell-h)-1$, then $v_1,v_2, \ldots, v_k, v_{k+1}$ have a common neighbor in $U_1$.
\end{claim}

\begin{proof}
	Since $|U_1|\ge nh-\ell+h+1$, the number of vertices in $U_1$ that are adjacent to all of $v_1,v_2, \ldots, v_k, v_{k+1}$ is at least \[nh-\ell+h+1-((k+1)(\ell-h)-1)-(2(\ell-h)-1)\ge nh- ((p+2)(\ell-h)-3)\ge 1\,.\]  Therefore, $v_1,v_2, \ldots, v_k, v_{k+1}$ have a common neighbor in $U_1$.
\end{proof}

Combining Claims~\ref{claim8} and \ref{claim10}, we know that $v_1$ and $v_2$ have a common neighbor in $U_1$. Furthermore, combining Claims~\ref{claim8}, \ref{claim9}, and \ref{claim10}, we know that $v_1$, $v_2$, and $v_3$ have a common neighbor in $U_1$. Continuing this process iteratively, we eventually obtain a vertex in $U_1$ that is adjacent to each of $v_1, v_2, \ldots, v_{p-1}$. This means that $G$ contains $K_p$ as a subgraph, resulting in a final contradiction.

%%%%%%%%%%%%%%%%%%%%%%%%%%%%%%%%%%%%%%%%%%
\section{Complete graphs versus fans}\label{secfan}
From the previous section we know that $r(K_p,F_n)=2(p-1)n+1$ holds for $n\ge (p^2-6)/2$ and $p\ge 3$. In this section we shall derive a slightly more refined bound of $n$. In other words, we prove that the equality holds for $n\ge (p^2-p-2)/2$ and $p\ge 3$. The lower bound has been proven in the previous section, so we focus on the upper bound. We will again resort to proof by contradiction, considering a graph $G$ with $2(p-1)n+1$ vertices. Suppose that neither $G$ contains a $K_p$ nor $\overline{G}$ contains an $F_n$. Similar to the previous proofs, we will arrive at a contradiction through a series of claims.

When $H=K_2$, we have $\ell=r(K_p, H)=p$. According to Property~\ref{property1}, we immediately obtain the following claim.

\begin{claim}\label{claim3.1}
	For every vertex $v$, we have $2n\le d_{\overline{G}}(v)\le 2n+p-3$.
\end{claim}

We employ the same approach as in Subsection~\ref{subsection3} and obtain the corresponding results. In other words, $U_1, \ldots, U_p$ form a partition of $V(G)$, $|U_p|\le p^{2}-3p+1$, and each $U_i$ is an independent set for $i\in [p-1]$. Without loss of generality, assume that $|U_1|\ge |U_2|\ge \cdots\ge |U_{p-1}|$. We can similarly establish that \[|U_1|\ge 2n-p+3.\]

What is more, we can derive a lower bound for $|U_2|$. Since $|U_2|\ge (|V(G)|-|U_p|-|U_1|)/(p-2)$, and $|U_1|\le 2n-1$, substituting these values yields \[|U_2|\ge 2n-p+2.\]

Similar to the previous section, we define a subset $U'_1$ in $V(G)\setminus (U_1\cup U_2)$, which consists of vertices that are not adjacent to any vertex of $U_1$. Additionally, we define another subset $U'_2$: for any vertex in $V(G)\setminus (U_1\cup U_2)$, if it is not adjacent to any vertex of $U_2$, then it belongs to $U'_2$. To avoid an $F_n$ in $\overline{G}$, we can also verify that $|U_i|+|U'_i|\le 2n$ for $i \in [2]$.

There are at least $2(p-3)n+1$ vertices in $G-(U_1\cup U_2\cup U'_1\cup U'_2)$. By the induction hypothesis, $G-(U_1\cup U_2\cup U'_1\cup U'_2)$ contains $K_{p-2}$ as a subgraph. We denote these $p-2$ vertices as $v_1,v_2, \ldots, v_{p-2}$. Next, we aim to prove that there exists a vertex from $U_1$ and a vertex from $U_2$ that, together with $v_1,v_2, \ldots, v_{p-2}$, form a complete graph $K_p$. The proofs for the following three claims are identical to those of Claims~\ref{claim8}, \ref{claim9}, and \ref{claim10}. We omit their proofs here.

\begin{claim}\label{claim11}
	For $k\in [p-2]$ and $i\in [2]$, the number of non-neighbors of $v_k$ in $U_i$ is at most $2p-5$.
\end{claim}

\begin{claim}\label{claim12}
	For $k\in [p-2]$ and $i\in [2]$, if $v_1,v_2, \ldots, v_k$ have a common neighbor in $U_i$, then the number of vertices in $U_i$ that are not adjacent to at least one of $v_1,v_2, \ldots, v_k$ is at most $(k+1)(p-2)-1$.
\end{claim}
	
\begin{claim}\label{claim13}
	For $k\in [p-3]$ and $i\in [2]$, if the number of vertices in $U_i$ that are not adjacent to at least one of $v_1,v_2, \ldots, v_k$ is at most $(k+1)(p-2)-1$, and the number of non-neighbors of $v_{k+1}$ in $U_i$ is at most $2p-5$, then $v_1,v_2, \ldots, v_k, v_{k+1}$ have a common neighbor in $U_i$.
\end{claim}

We denote the set of common neighbors of $v_1,v_2, \ldots, v_{p-2}$ in $U_1$ as $X$, and the set of common neighbors of $v_1,v_2, \ldots, v_{p-2}$ in $U_2$ as $Y$. According to the three claims above, we have $|X|\ge p$ and $|Y|\ge p-1$. If there is an edge between $X$ and $Y$, then the two endpoints of this edge, together with $v_1,v_2, \ldots, v_{p-2}$, form a complete graph $K_p$ which leads to a contradiction. If there is no edge between $X$ and $Y$, then there exists a vertex $x$ in $U_1$ that is not adjacent to at least $2n-p+2$ vertices in $U_1$ and not adjacent to at least $p-1$ vertices in $U_2$. Since both $U_1$ and $U_2$ are independent sets, these $2n+1$ vertices that are adjacent to $x$ in $\overline{G}$ must contain an $nK_2$ as a subgraph. Thus, we obtain an $F_n$ in $\overline{G}$, resulting in a final contradiction.

\section*{Acknowledgments}

%We are grateful to the anonymous referees for their very careful comments.

The first author was partially supported by NSFC under grant numbers 11601527 and 11971011, while the second author was partially supported by NSFC under grant numbers 12161141003 and 11931006.

\end{document}